\documentclass[11pt,reqno]{amsart}

\usepackage{amsbsy}
\usepackage{amscd}
\usepackage{amsfonts}
\usepackage{amsgen}
\usepackage{amsmath}
\usepackage{amsopn}
\usepackage{amssymb}
\usepackage{amstext}
\usepackage{amsthm}
\usepackage{amsxtra}
\usepackage[all]{xy}
\usepackage{comment}
\usepackage{euscript}

\usepackage{csquotes} %%% for quotation
\usepackage{physics}

\usepackage{mathrsfs}
\usepackage{MnSymbol}
\usepackage{rotating}
\usepackage{url}

\usepackage{mathtools}

\usepackage{algpseudocode,algorithm}

\usepackage{ifthen}

\usepackage{graphicx}
\usepackage{hyperref}

\hypersetup{% hyperref options
	setpagesize=false,
	bookmarksnumbered=true,%
	bookmarksopen=true,%
	colorlinks=true,%
	linkcolor=blue,
	citecolor=red,
}
\theoremstyle{plain}
\newtheorem{thm}{Theorem}[section]
\newtheorem{prop}[thm]{Proposition}
\newtheorem{lem}[thm]{Lemma}
\newtheorem{cor}[thm]{Corollary}

\theoremstyle{definition}\newtheorem{defn}[thm]{Definition}

\newtheorem{example}[thm]{Example}

\newtheorem{note}[thm]{Notation}

\numberwithin{equation}{section}

\renewcommand{\theta}{\vartheta}
\renewcommand{\phi}{\varphi}
\renewcommand{\epsilon}{\varepsilon}
\renewcommand{\subset}{\subseteq}

\newcommand{\mc}[1]{\mathcal{#1}}
\newcommand{\mf}[1]{\mathfrak{#1}}

\newcommand{\mr}[1]{\mathrm{#1}}
\newcommand{\mbb}[1]{\mathbb{#1}}

\newcommand{\N}{\mathbb N}

\newcommand{\R}{\mathbb R}
\newcommand{\C}{\mathbb C}

\newcommand{\K}{\mathbb K}

\newcommand{\NC}{\mr{NC}}

\newcommand{\mcR}{\mathcal R}

\newcommand{\hilb}{\mc{H}}

\DeclareMathOperator*{\id}{id}
\DeclareMathOperator*{\Real}{Re}
\DeclareMathOperator*{\Image}{Im}

\newcommand{\eps}{\epsilon}

\newcommand{\ccop}[2]{G_{#1}^{#2}}

\newcommand{\HP}{\mathbb{H}^+}
\newcommand{\HM}{\mathbb{H}^-}

\newcommand{\SPN}{\mathrm{SPN}}
\newcommand{\CW}{\mathrm{CW}}
\newcommand{\sa}{\mathrm{s.a.}}

\newcommand{\cptprob}{\mc{B}_{c}(\R)} %for the set of compact support probability measure
\newcommand{\D}{\mf{D}_2}

\newcommand{\muSPNB}[2]{{\mu_\SPN^\Box(#1, #2)}}

\usepackage{todonotes}

\begin{document}

%%%------To switch dvipdf-chain and pdf-chain
\newboolean{DVIPDF}
\setboolean{DVIPDF}{false}

\mathtoolsset{showonlyrefs=true}

%------------------------------------------%%%

%%%------To switch  paper or else
\newboolean{PHD}
%\setboolean{PHD}{true}
\setboolean{PHD}{false}
\newcommand{\PHD}{\boolean{PHD}}

\title[Identifiability of PRM]{
Identifiability of Parametric Random Matrix Models
}
\author{Tomohiro Hayase}
\address{Graduate School of Mathematical Sciences, University of Tokyo, 3-8-1 Komaba, Meguro-ku, Tokyo, 153-8914, Japan }
\email{\href{mailto:}{hayase@ms.u-tokyo.ac.jp}}
\date{\today}

\begin{abstract}
We investigate parameter identifiability of spectral distributions of random matrices.
In particular, we treat compound Wishart type and  signal-plus-noise type. 
We show that each model is identifiable up to some kind of rotation of parameter space.
Our method is based on free probability theory.
\end{abstract}

\maketitle

\tableofcontents
\section{Introduction}

%\bpara{Identifiability}\hfill
Identifiability analysis is fundamental in a theoretical understanding of statistical models, for example,  log-likelihood maximization.
A parametric statistical model  $(P_\theta)_{\theta \in \Theta}$, a parametric family  of probability measures,  is said to be \emph{identifiable} if the map $\theta \mapsto P_\theta$ is injective. 
For a statistical model, its identifiability is necessary for its regularity.
Under regularity condition, then maximal likelihood estimator has a good behavior such as asymptotic normality.
In general, a geometry of log-likelihood is determined by the Fisher information matrix (see  \cite{amari2016information}), which is  expected  Hessian of log-likelihood with respect to parameters.
If a statistical model is non-identifiable, then  the Fisher information matrix is singular, and the eigenspace for the zero eigenvalue is determined by non-identifiable parameters.
Therefore, determining non-identifiable parameters is important in   non-identifiable models.

%\bpara{Random Matrix Models}\hfill

In this paper, we investigate identifiability of statistical models introduced for parameter estimation of random matrix models. 
In \cite{hayase2018cauchy}, two typical random matrix models,  the compound Wishart model $W_\CW$ and the signal-plus-noise model  $W_\SPN$ are treated.
They are defined as the following: 
\begin{align}
W_\CW(D) &=  Z^*D Z,\\
W_\SPN(A, \sigma) &=  ( A + \sigma Z)^* (A + \sigma Z),
\end{align}
where $Z$ is $p \times d$  matrix of independent and identically distributed Gaussian random variables with mean zero and variance $1/d$,
$D$ and $A$ are deterministic matrices, and $\sigma \in \R$.
For any self-adjoint matrix $W$, 
let us denote by $\mu_W$ the eigenvalue distribution defined as 
\begin{align}
\mu_W=\frac{1}{d}\sum_{k=1}^{d} \delta_{\lambda_k},
\end{align}
where $\lambda_1 \leq \lambda_2 \dots \leq \lambda_d$ are eigenvalues of $W$.
The parameter estimation method introduced in \cite{hayase2018cauchy} is minimizing modified KL-divergence between a statistical model 
\begin{align}
&\mu^\Box_\CW(D), \ D \in M_{p}(\C)\\
&(\text{resp.\,}  \muSPNB{A}{\sigma}, \  A \in M_{p,d}(\C),\  \sigma \in \R )
\end{align}
and a sample of the empirical eigenvalue distribution $\mu_{W_\CW(D_0)}$ (resp.\,$\mu_{W_\SPN(A_0, \sigma_0)}$), where true parameters $D_0, A_0$, $\sigma_0$ are unknown. The definition of the statistical models $\mu^\Box_x$ is based on \emph{free deterministic equivalent}. 
The free deterministic equivalent is introduced by \cite{speicher2012free}, which is a deterministic and infinite-dimensional approximation of random matrices  based on a central limit theorem of  the eigenvalue distribution.
%\bpara{Unidentifiable}\hfil

It directly follows from the definition of  $\mu_x^\Box$ that 
\begin{align}
\mu_D = \mu_{D^\prime} %
&\Longrightarrow   %
\mu^\Box_\CW(D) = \mu^\Box_\CW(D^\prime),\\
\mu_{A} = \mu_{A^\prime}, \ \sigma^2 = \sigma^{^\prime 2} %
&\Longrightarrow %
\mu^\Box_\SPN(A,\sigma) = \mu^\Box_\SPN(A^\prime,\sigma^\prime).
\end{align}
In particular, these statistical models are not identifiable.
For the CW model, it is easy to show that the converse also holds:
\begin{align}
\mu_D = \mu_{D^\prime} %
&\iff %
\mu^\Box_\CW(D) = \mu^\Box_\CW(D^\prime).
\end{align}
In other words, if we replace the parameter set by the set of eigenvalue distributions then this model becomes identifiable.
Note that there is a bijection between the set of eigenvalue distributions 
and 
\begin{align}
\{ v \in \R^p \mid v_k = v_{\pi(k)}, \forall k= 1, \dots, p, \ \forall \pi \in S_p \},
\end{align}
where $S_p$ is the permutation group of $p$ elements.
However, it is not clear that the converse holds for the SPN model.

%\bpara{Main Result}\hfill

The main theorem of this paper is as follows.
\begin{thm}
    Let $p, d \in \N$ with $p \geq d$. For $A,B \in M_{p,d}(\C)$ and $\sigma, \rho \in \R$, the following holds:
    \begin{align}
    \mu^\Box_\SPN(A,\sigma) = \mu^\Box_\SPN(B,\rho)  \iff
    \begin{cases}
    \mu_{A^*A} = \mu_{B^*B}, \\
    \sigma^2 = \rho^2.
    \end{cases}  
    \end{align}
\end{thm}

In particular, if we replace the parameter space by the direct product of 
singular value distribution and the nonzero real numbers, 
then this statistical model becomes identifiable.
Note that there is a bijective between the direct product and 
\begin{align}
\{ v \in \R^d \mid v_{\pi(k)} = v_k \geq 0 \ \forall k =1, \dots, d, \forall \pi \in S_d \}  \times  \{ v \in \R \mid v \geq 0 \}.
\end{align}
Our proof  consists of  
an analytic part based on operator-valued analytic free additive subordination \cite{belinschi2013analytic} and a combinatorial part based on free multiplicative deconvolution \cite{ryan2007free, ryan2multiplicative}.

\section{Related Works}
%random matrices
The compound  Wishart random matrix was introduced by  \cite{speicher1998combinatorial}.
It appears as sample covariance matrices of correlated samplings
\cite{burda2011applying, couillet2011deterministic, hasegawa2013random}.
%See \cite{hiai2006semicircle}  for more detail.
The signal-plus-noise random matrix appears in  signal precessing \cite{ryan2007free, hachem2012large, vallet2012improved}. 

% free deterministic equivalents
Free probability is invented by Voiculescu \cite{voiculescu1985symmetries}.
In free probability theory, motivated by solving a problem in operator algebras, some infinite-dimensional operators are described as  infinite-dimensional limit of random matrices.
The approximation is  based on a central limit theorem, which is called the \emph{free central limit theorem}, of eigenvalue distribution of random matrices  \cite{voiculescu1991limit}.
Conversely, the purpose of free deterministic equivalent is to approximate fixed-size but large random matrix models by deterministic operators.

% identifiability
For analysis of non-identifiable models, generic identifiability was introduced in \cite{allman2009identifiability}.

\section{Preliminary}
\subsection{Freeness}\label{ssec:freeness}\hfill

First, we summarize some definitions from operator algebras and free probability theory. See \cite{Mingo2017free} for the detail.
\begin{defn}\hfill
	\begin{enumerate}
		\item A \emph{C$^*$-probability space}  is a pair $(\mf{A}, \tau)$  satisfying followings.
		\begin{enumerate}
			\item The set $\mf{A}$ is \emph{a unital $C^*$-algebra}, that is, a possibly non-commutative subalgebra of the algebra  $B(\hilb)$ of bounded $\C$-linear operators on a Hilbert space $\hilb$ over $\C$ satisfying the following conditions:
			\begin{enumerate}
				\item it is stable under the adjoint $* : a \to a^*, a \in \mf{A}$,
				\item it is closed under  the topology of the operator norm of $B(\hilb)$,
				\item it contains the identity operator $\id_\hilb$  as the unit $1_\mf{A}$ of $\mf{A}$.
			\end{enumerate}
			\item The function $\tau$ on  $\mf{A}$ is a \emph{faithful tracial state}, that, is  a $\C$-valued  linear functional with
			\begin{enumerate}
				\item $\tau(a)\geq 0$ for any $a \geq 0$, and the equality holds if and only if $a=0$,
				\item $\tau(1_\mf{A})=1$,
				\item $\tau(ab)=\tau(ba)$ for  any $a, b \in \mf{A}$.
			\end{enumerate}
		\end{enumerate}
		\item A  subalgebra $\mf{B}$ of a C$^*$-algebra $\mf{A}$ is called a \emph{$*$-subalgebra} if it is stable under the adjoint operator $*$. Moreover, it is called a \emph{unital C$^*$-subalgebra} if the $*$-subalgebra is closed under the operator norm topology and contains $1_\mf{A}$ as its unit.
		\item Two unital $C^*$-algebras are called \emph{$*$-isomorphic} if there is a bijective linear map between them which preserves the $*$-operation and the multiplication.
		\item Let us denote by $\mf{A}_\sa$ the set of \emph{self-adjoint} elements, that is, $a = a^*$ of $\mf{A}$.
		\item Write $\Real a := (a + a^*)/2$ and $\Image a := ( a- a^*)/{2i}$ for any $a \in \mf{A}$.
		\item The \emph{distribution} of  $a \in \mf{A}_\sa$  is the probability measure $\mu_a \in \cptprob$ determined by
		\[
		\int x^k \mu_a(dx) = \tau(a^k),\  k \in \N.
		\]
		\item For $a \in \mf{A}_\sa$, we define its Cauchy transform $G_a$ by $G_a(z):= \tau[ (z -a )^{-1} ] \ (z \in \C \setminus \R)$, equivalently, $G_a:=G_{\mu_a}$.
		
	\end{enumerate}
	
\end{defn}

\begin{defn}
	A family of $*$-subalgebras $(\mf{A}_j)_{j \in J}$ of $\mf{A}$ is said to be \emph{free}
	if the following factorization rule holds: for any $n \in \N$ and indexes $j_1, j_2, \dots, j_n \in J $ with $j_1 \neq j_2 \neq j_3  \neq \cdots \neq j_n$, and $a_l \in \mf{A}_l$ with $\tau(a_l)=0$ $(l = 1 ,\dots, n)$, it holds that
	\[
	\tau(a_1 \cdots a_l) = 0.
	\]
	Let  $(x_j)_{j \in J}$ be a family of self-adjoint elements $x_j \in \mf{A}_\sa$. For $j \in J$, let $\mf{A}_j$ be the $*$-subalgebra of polynomials of $x_j$.
	Then $(x_j)_{j \in J}$ is said to be  free if $\mf{A}_j$  is free.
\end{defn}

We introduce special elements in a non-commutative probability space.
\begin{defn}
	Let $(\mf{A}, \tau)$  be a C$^*$-probability space.
	\begin{enumerate}
		\item
		An element $s \in \mf{A}_\sa$ is called \emph{standard semicircular} if its distribution is given by the standard semicircular law;
		\[
		\mu_s(dx) = \frac{\sqrt{4-x^2}}{2\pi}{\bf 1}_{[-2,2]}(x)dx,
		\]
		where ${\bf 1}_S$ is the indicator function for any subset $S \subset \R$.
		\item Let $v > 0$.    An element $ c \in \mf{A}$ is called \emph{circular of variance  $v$} if
		\[c=\sqrt{v}\frac{s_1+is_2}{\sqrt{2}},\]
		where $(s_1,s_2) $ is a pair of free standard semicircular elements.
        In addition. $c$ is called \emph{standard circular element} if $v=1$. 
		\item A \emph{$*$-free circular family} (resp.\,standard $*$-free circular family) is a family $\{ c_j \mid j \in J \}$ of circular elements $c_j \in \mf{A}$ such that $\bigcup_{j\in J}\{ \Real c_j , \Image c_j \}$ is free (resp. and each elements is of variance $1$).
	\end{enumerate}

\end{defn}

\bigskip
\begin{defn}
	Let $(\mf{A}, \tau)$ be a C$^*$-probability space and $\mf{B}$ be a unital  C$^*$-subalgebra  of $\mf{A}$. Recall that they share the unit: $I_\mf{A} = I_\mf{B}$.
	\begin{enumerate}
		\item
		Then a linear operator $E \colon \mf{A} \to \mf{B}$ is called a
		\emph{conditional expectation onto $\mf{B}$} if it satisfies following conditions;
		\begin{enumerate}
			\item $E[b] = b$ for any $b \in \mf{B}$,
			\item $E[b_1 a b_2] = b_1E[a]b_2$ for any $a \in \mf{A}$ and $b_1, b_2 \in \mf{B}$,
			\item $E[a^*] = E[a]^*$ for any $a \in \mf{A}$.
		\end{enumerate}
		
		\item
		We write $\HP(\mf{B}) := \{ W \in \mf{B} \mid \text{ there is $\eps > 0$ such that $\Im W \geq \eps I_\mf{A}$} \}$ and $\HM(\mf{B}) := - \HP(\mf{B})$.
		\item
		Let $E \colon \mf{A} \to \mf{B}$ be a conditional expectation.
		For $a \in \mf{A}_\sa$, we define a \emph{$E$-Cauchy transform} as the map $\ccop{a}{E} \colon \HP(\mf{B}) \to \HM(\mf{B})$, where
		\[
		\ccop{a}{E}(Z) := E[ (Z - a)^{-1}],  \ Z \in \HP(\mf{B}).
		\]
		If there is no confusion, we also call $E$ a $\mf{B}$-valued Cauchy transform.
	\end{enumerate}

\end{defn}

\begin{defn}(Operator-valued Freeness)
	Let $(\mf{A}, \tau)$ be a C$^*$-probability space,  and $E :  \mf{A} \to \mf{B}$ be a conditional expectation.
	Let  $(\mf{B}_j)_{j \in J}$ be a family of $*$-subalgebras of $\mf{A}$ such that $\mf{B} \subset \mf{B}_j$.
	Then $(\mf{B}_j)_{j \in J}$  is said to be \emph{$E$-free}
	if the following factorization rule holds: for any $n \in \N$ and indexes $j_1, j_2, \dots, j_n \in J $ with $j_1 \neq j_2 \neq j_3  \neq \cdots \neq j_n$, and $a_l \in \mf{B}_l$ with $E(a_l)=0$ $(l = 1 ,\dots, n)$, it holds that
	\[
	E(a_1 \cdots a_l) = 0.
	\]
	In addition, a family of elements $X_j \in \mf{A}_\sa \ ( j \in J)$ is called $E$-free if the family of $*$-subalgebra of the $\mf{B}$-coefficient polynomials of $X_j$ is $E$-free.
\end{defn}

\subsection{Random Matrix Models and Free Deterministic Equivalents}

\begin{defn}
    Fix a probability measure space $(\Omega, \mf{F}, \mbb{P})$. 
    Write $\mbb{E}[\cdot] = \int \ \cdot \ \mbb{P}(d\omega)$.
    Let $p, d \in \N$.
    Then real (resp.\,complex) $p \times d$ Ginibre random matrix of variance $v>0$  is  defined as $p \times d$ matrix of independent and  identically distributed  real (resp.\,complex) Gaussian  random variables  $Z_{ij}$ $(i=1,\dots, p, j=1, \dots,d )$ such that 
    \begin{align}
    \mbb{E}[Z_{ij}] = 0,  	 \mbb{E}[\bar{Z}_{ij}Z_{ij}] = v.
    \end{align}
\end{defn}

\begin{defn}
    Let $\K = \R $ (resp.\,$\K=\C$).
	Let us denote by $Z$ the real (resp.\,complex) $p \times d$ Ginibre random matrix of variance $1/d$.
	\begin{enumerate}
	\item A real (resp.\,complex) \emph{compound Wishart model} ( CW model for short) of type $(p,d)$ is defined as a 
	parametric family $W_\CW,$ where
	\begin{align}
	 W_\CW (D) := Z^*DZ, \  D \in M_{p}(\K).
	\end{align} 
	\item 
	A real (resp.\,complex)  \emph{signal-plus-noise model} (SPN model for short ) of type $(p,d)$ is defined as a parametric family $W_\SPN$, where
	\begin{align}
	W_\SPN(A,\sigma):=(A+\sigma Z)^*(A+\sigma Z),%
	\ A \in M_{p,d}(\K), \ \sigma \in \R.
	\end{align}
	\end{enumerate}
\end{defn}

Here we introduce free deterministic equivalent of each random matrix model. Note that the free deterministic equivalent does not depend on the choice of the field $\R$ or $\C$.

\begin{defn}
	Let $p,d \in \N$. Fix a C$^*$-probability space $(\mf{A}, \tau)$. Let us denote by $C$ the $p \times d$ matrix of $*$-free circular elements in $(\mf{A}, \tau)$ so that
	\begin{align}
	\tau(C_{ij})=0, \ \tau(C_{ij}^* C_{ij}) = 1/d.
	\end{align}
	
	\begin{enumerate}
		\item The free deterministic equivalent of CW model (FDECW model, for short) of type $(p,d)$ is defined as a parametric family $W^\Box_\CW$, where 
		\begin{align}
		W^\Box_\CW(D) = C^*DC,  \  D \in M_p(\C).
		\end{align}
		In addition, we denote by $\mu_\CW(D)$ the distribution of $W^\Box_\CW(D)$
		in the C$^*$-probability space $(M_d(\mf{A}), \tr_d \otimes \tau)$:
		\begin{align}
		\mu^\Box_\CW(D) = \mu_{W^\Box_\CW(D)}.
		\end{align}
		
		\item The free deterministic equivalent of SPN model (FDESPN model, for short) of type $(p,d)$ is defined as a parametric family $W^\Box_\SPN$, where 
		\begin{align}
		W^\Box_\SPN(A, \sigma)=(A + \sigma C)^* (A + \sigma C),
		\ A \in M_{p,d}(\C), \ \sigma \in \R.
		\end{align}
		In addition we denote by $\mu_\SPN(A, \sigma)$ the distribution of $W^\Box_\SPN(A, \sigma)$ in the C$^*$-probability space $(M_d(\mf{A}), \tr_d \otimes \tau)$, that is,
		\begin{align}
		\mu^\Box_\SPN(A, \sigma) = \mu_{W^\Box_\SPN(A, \sigma)}.
		\end{align}
	\end{enumerate}	
\end{defn}

\newcommand{\MP}[1]{\mu_{\mathrm{MP}(#1)}}
\newcommand{\boxdiag}{\ \frame{$\smallsetminus$}\;}

\section{Identifiability}

\ifthenelse{\PHD}{Write Intro of Identifiability}{}
\begin{comment}
In statistics, a parametric family  $(P_\theta)_{\theta \in \Theta}$ of probability measures with a set  $\Theta$ of parameters is said to be identifiable if the map $\theta \mapsto P_\theta$ is injective.

In parameter estimation of a random matrix, we consider a spectral distribution of its free deterministic equivalent.
However, the parametric family of spectral distributions is possibly non-identifiable because eigenvalue distribution of deterministic matrices is invariant under permutation of indexes. 
\end{comment}

\begin{comment}

To investigate the parametric spectral distributions, we introduce the following version of identifiability.
\begin{defn}
For a random matrix model $W_x(\theta)$ ($x=\CW, \SPN$) is said to be \emph{spectral identifiable} up to a relation $\sim$ of $\Theta$  if 
$ \mu_{W_x^\Box(\theta_1)} =  \mu_{W_x^\Box(\theta_2)}$ if and only if $\theta_1 \sim  \theta_2$ for any $\theta_1, \theta_2 \in \Theta$.
\end{defn}
\end{comment}

\subsection{Identifiability of CW Model}\hfill

First, we quickly check the identifiability of the CW model.
Fix $p, d \in \N$. Let $D, D^\prime \in M_p(\C)]_\sa$ and $v=(v_1 \leq v_2 \leq \dots v_p), v^\prime=(v^\prime_1 \dots v^\prime_d) \in \R^p$ be the vectors of  eigenvalues of $D, D^\prime$ respectively.
Assume that 
\begin{align}
\mu_\CW^\Box(D) = \mu_\CW^\Box(D^\prime). \label{align:assump-cw}
\end{align}
Now since $\mu_\CW^\Box(D)$ is a compound free Poisson law ( see \cite{nica2006lectures}), the $\mcR$-transform of $\mu_\CW^\Box(D)$ is given by the following.
\begin{align}
\mcR(b,v) =  \frac{1}{d}\sum_{k=1}^{p} \frac{v_k}{1- v_k b}, \ b \in \HM(\C).
\end{align}
By the assumption \eqref{align:assump-cw}, it holds that 
\begin{align}
\mcR(b, v ) = \mcR( b, v^\prime),  b \in \HM(\C).
\end{align}
Since all polos of $\mcR(\cdot, v)$ are  order one,  $v$ and $v^\prime$ are equal up to permutation of entries, that is, there is a permutation $\pi \in S_p$ such that 
\begin{align}
v_{\pi(k) } = v_k^\prime, \ k = 1, \dots, p. 
\end{align} 
Equivalently, we have
\begin{align}
\mu_D = \mu_{D^\prime}. 
\end{align}

\subsection{Identifiablity of SPN Model}\hfill

Next, we work on the SPN model.
We prove the following identifiability of the statistical model $\mu^\Box_\SPN$ for the random matrix model $W_\SPN$.
The proof is divided into an analytic part and a combinatorial one.
\begin{thm}\label{thm:identifiability-spn}
	Let $p,d \in \N$ with $p \geq d$, $A, B \in M_{p,d}(\C)$, and $\sigma, \rho \in \R$. Then 
	$\muSPNB{A}{\sigma} = \muSPNB{B}{\rho}$ if  and only if $\mu_{A^*A} = \mu_{B^*B}$ and $\sigma^2 = \rho^2$.
\end{thm}
The proof is  postponed to  Section~\ref{sssec:proof-of-iden}.

\subsubsection{Analytic Part}\hfill

Write
\begin{align}
\D =  \left\{\begin{bmatrix}
z_1I_d & 0 \\
0 & z_2I_p
\end{bmatrix}
\mid z_1,  z_2 \in \C \right\} \subset M_{p+d}(\C) \subset M_{p+d}(\mf{A}).
\end{align}
We identify $\D$ and $\C^2$ via the following isomorphism $\D \simeq \C^2$:
\begin{align}
\begin{bmatrix}
z_1I_d & 0 \\
0 & z_2I_p
\end{bmatrix}
\mapsto 
\begin{bmatrix}
z_1 \\
z_2
\end{bmatrix}.     
\end{align} 
We define a conditional expectation $E \colon M_{p+d}(\mf{A}) \to \C^2$ by
\begin{align}
E(X) =  \begin{bmatrix}
\tr_{d} \otimes \tau(X_{ +  ,+ }) \\
\tr_{p} \otimes \tau(X_{-,-})
\end{bmatrix},
\end{align}
where $X_{+,+} \in M_d(\mf{A})$ is the $d \times d$-upper left corner of $X \in M_{p,d}(\mf{A})$ and $X_{-,-} \in M_p(\mf{A})$ is the $p \times p$-lower right corner of $X$.
For $X \in M_{p+d}(\mf{A})$ and $z \in \HP(\C^2) = \{ (z_1, z_2) \in \C^2 \mid \Im z_1, \Im z_2 > 0\}$, we write 
\begin{align}
G_{X}(z) &= E[ (z - X)^{-1} ],\\
h_X(z) &= G_X(z)^{-1} - z.
\end{align}
For any rectangular matrix $Y \in  M_{p,d}(\mf{A})$, write  
\begin{align}
\Lambda(Y) = %
\begin{bmatrix}
0 &  Y^* \\
Y &  0
\end{bmatrix}.
\end{align}
Let $z= (\alpha, \beta) \in \C^2$.
Then we have
\begin{align}
\left(z - \Lambda(Y) \right)^{-1} =%
\begin{bmatrix}
\alpha I_d & -Y^* \\
-Y & \beta I_p
\end{bmatrix}^{-1}
=\begin{bmatrix} \beta (\alpha \beta I_d - Y^*Y )^{-1}  & Y^*( \alpha \beta I_p - YY^* )^{-1}\\
( \alpha \beta I_p - YY^*)^{-1} Y & \alpha ( \alpha \beta I_p - YY^* )^{-1} \end{bmatrix}.
\end{align}
Applying $E$, we have
\begin{align}\label{align:temp}
G_{\Lambda(Y)}(z) = \begin{bmatrix}
\beta \tr_d \otimes \tau [(\alpha  \beta I_d - Y^*Y )^{-1}]\\
\alpha \tr_p \otimes \tau [( \alpha \beta I_p - YY^* )^{-1}]
\end{bmatrix}.
\end{align}
In particular, $G_{\Lambda(Y)}$ is determined by $\mu_{Y^*Y}$.
Let $C \in M_{p,d}(\mf{A})$ be a matrix of $*$-free standard circular elements.
By \ifthenelse{\PHD}{Proposition~\ref{prop:d2-freeness}}{\cite[Proposition~5.30]{hayase2018cauchy}},
$\Lambda(C)$ is a $\C^2$-valued semicircular element (see \cite[Section~9.1]{Mingo2017free} for the definition)  with 
the following variance mapping $\eta \colon \C^2 \to \C^2$:
\begin{align}
\eta \left( \begin{bmatrix}
x \\
y
\end{bmatrix} \right)=&
\begin{bmatrix}
(p/d) y \\
x 
\end{bmatrix}.
\end{align}
Hence the following equations hold for any $z \in \HP(\C^2)$:
\begin{align}
G_{\sigma \Lambda(C)}(z)^{-1} &=   z -  \sigma^2 \eta(  G_{\sigma \Lambda(C)}(z) ),\\
h_{\sigma\Lambda(C)}(z)  &=  \sigma^2 \eta(  G_{\sigma \Lambda(C)}(z) ).\label{align:h-trans}
\end{align}

Next, to prove a key lemma, we refer to an analytic free additive subordination formula based on \cite{belinschi2013analytic}.
\begin{cor}\label{cor:a+s}
	Set $a := \Lambda(A)$ and $s := \sigma\Lambda(C)$.
	Then there exists a  pair of Fr\'eche analytic (equivalently, holomorphic) mappings $\psi_1, \psi_2 \in \mr{Hol}(\HP(\C^2))$ so that 
    for all $z \in \HP(\C^2)$, 
    \begin{align}
	\Im \psi_j(z) &\geq \Im z, \forall j \in \{ 1,2 \}, \label{align:a+s-1} \\ 
	h_{a}(\psi_1(z)) + z &= \psi_2(z), \label{align:a+s-2}\\
	h_{s}(\psi_2(z)) + z &= \psi_1(z), \label{align:a+s-3}\\
	G_{a + s}(z) &= G_{a}(\psi_1(z)),  \mr{ \ and, }	 \label{align:a+s-4}\\
    G_{a + s}(z) &= G_{s}(\psi_2(z)).\label{align:a+s-5}
	\end{align}
\end{cor}

\begin{proof}
	By 
	\ifthenelse{\PHD}{Proposition~\ref{prop:d2-freeness}}{\cite[Proposition~5.30]{hayase2018cauchy}}, 
	the pair $(a,s)$ is $E$-free. 
	Then the assertion follows from \cite[Theorem~2.7]{belinschi2013analytic}.
\end{proof}

\begin{lem}\label{lem:subord}
	Let $p,d \in \N$ with $p\geq d$. Let $A \in M_{p,d}(\C)$ and $\sigma \in \R$. 
	Then   we have the following equation between holomorphic mappings on $\HP(\C^2)$:
	\begin{align}
	G_{\Lambda(A + \sigma C)}(z) =
	G_{\Lambda(A)}\left[  \sigma^2 \eta\left(G_{\Lambda(A + \sigma C)} ( z )  \right) + z \right], \forall z \in \HP(\C^2).
	\end{align}
\end{lem}

\begin{proof}
	Set  $a := \Lambda(A)$ and $s:= \Lambda(C)$.
	Pick same holomorphic mappings $\psi_1$ and  $\psi_2$ as in Corollary~\ref{cor:a+s}. Then for any  $z \in \HP(\C^2)$,
	\begin{align}
	G_{a + s}(z) &= G_{a} ( \psi_1(z) )%
	&\text{(by \eqref{align:a+s-4})}\\
	&= G_{a} \left( h_{s}(\psi_2(z)) + z \right )%
	&\text{(by \eqref{align:a+s-3})}\\ 
	&= G_{a} \left( \sigma^2\eta(G_{s}(\psi_2(z) )) + z \right )%
	&\text{(by \eqref{align:h-trans})}\\
	& = G_{a} \left( \sigma^2\eta(G_{a+s}(z)) + z \right ).%
	&\text{(by \eqref{align:a+s-5})}
	\end{align}
\end{proof}

Now we have prepared to prove the first key lemma.
\begin{lem}\label{lem:key-analy}
	Fix $p,d \in \N$ with $p \ge d$. Let $A, B \in M_{p,d}(\C)$ and $\sigma \in \R$.
	If $\muSPNB{A}{\sigma}  = \muSPNB{B}{0}$ then $\sigma = 0$.
\end{lem}
\begin{proof}

	Assume that $\muSPNB{A}{\sigma}  = \muSPNB{B}{0}$.
	Then $\ccop{\Lambda(A+\sigma C)}{}=\ccop{\Lambda(B)}{}$ since $G_{\Lambda(Y)}$ is determined by $\mu_{Y^*Y}$ for any $Y \in M_{p,d}(\mf{A})$.
	
	In the case $B = 0$, it holds that  $(A + \sigma C)^*(A+\sigma C)=0$.
	Thus $A = - \sigma C$ and $A^*A = \sigma^2 C^*C$.  Since $\mu_{C^*C}$ has no atom and $\mu_{A^*A}$ is a sum of delta measures,  we have $\sigma = 0$.
	
	Consider the case $B \neq 0$. Write $\beta:= \norm{B^*B}^{1/2} > 0$.
	Now for any  $z \in \HP(\C^2)$,  by the assumption and Lemma~\ref{lem:subord}, the following holds:
	\begin{align}
	G_{\Lambda(A)} \left[  \sigma^2 \eta \left( G_{\Lambda(B)}(z) \right) + z \right] =%
	G_{\Lambda(B)}(z), z \in \HP(\C^2). \label{align:key-sub}
	\end{align}
	Let
	\begin{align}
	g(z) &:=G_{\Lambda(B)}(z),\\
	f(z) &:= (z_2 z_1 - \beta^2)G_{\Lambda(B)}(z).
	\end{align}
	Then  
	\begin{align}\label{align:degree-one}
	\lim_{\gamma \to +0} f( \beta + i\gamma,\beta + i\gamma) = ( \frac{m\beta}{d}, \frac{m \beta}{p}) \neq 0,
	\end{align}
	where $m \geq 1$ is the multiplicity of the eigenvalue $\beta$ of $\sqrt{B^*B}$.
	Let $a_1 \leq \dots \leq a_d$ be eigenvalues of $\sqrt{A^*A}$. Then 
	\begin{align}
	G_{\Lambda(A)}\left( z_1, z_2 \right) =  (\frac{z_2}{d}\sum_{k=1}^d \frac{1}{z_2z_1 - a_k^2}, %
	\frac{z_1}{p} \sum_{k=1}^{d}\frac{1}{z_2z_1 - a_k^2} + \frac{p-d}{pz_2}).
	\end{align}
	Now for any $k=1,\dots,d$ and $j=1,2$,
	\begin{align}\label{align:rational}
	\frac{g(z)_j}{ g(z)_1g(z)_2 - a_k^2} = \frac{f(z)_j}{f(z)_1f(z)_2 - a_k^2(z_1z_2-\beta^2)^2} (z_1z_2-\beta^2).
	\end{align}
	Let $\gamma >0$ and $z_1=z_2=\beta+i\gamma$. Then \eqref{align:rational} converges to 
	$0$ as $\gamma \to +0$ by \eqref{align:degree-one}.
	
	Assume that $\sigma \neq 0$,  then by $\eqref{align:rational}$, it holds that
	\begin{align}
	\lim_{ \substack{ z = (\beta + i \gamma,\beta + i \gamma)\\ \gamma \to +0} } G_{\Lambda(A)} \left[\sigma^2 \eta\left( G_{\Lambda(B)}(z) \right)  + z \right] =  (0,  \frac{p-d}{p\beta}).
	\end{align}
	In particular,
	\begin{align}
	\lim_{ \substack{ z = (\beta + i \gamma,\beta + i \gamma)\\ \gamma \to +0} } (z_1z_2-\beta^2)G_{\Lambda(A)} \left[\sigma^2 \eta\left( G_{\Lambda(B)}(z) \right)  + z \right] =  0.
	\end{align}
	By  \eqref{align:key-sub}, this contradicts $\eqref{align:degree-one}$.
	Therefore $\sigma = 0$.
\end{proof}

\subsubsection{Combinatorial Part}\hfill

\newcommand{\boxconv}{ \framebox[7pt]{$\star$} }
\newcommand{\fpower}{\Xi}
\newcommand{\coef}[3]{ \mathrm{Cf}_{#2#3}(#1) }
\newcommand{\fzeta}{\mathrm{Zeta}}

We use the free multiplicative deconvolution introduced by \cite{ryan2multiplicative,ryan2007free}.
We quickly review the deconvolution.

First, we introduce a family of formal power series, since the deconvolution is defined as an operation between moment power series.
Let us denote by $\fpower$ the set of formal power series without the constant term of the form
\begin{align}\label{align:power}
f(z) = \sum_{n=1}^\infty \alpha_n z^n, 
\end{align}
with $\alpha_n \in \C (\forall n \in \N)$.
Let $f  \in \fpower$ be as in  \eqref{align:power}. For every $n \in \N$ we denote
\begin{align}
\coef{f}{n}{} = \alpha_n.
\end{align}

Second, we introduce Kreweras complement and boxed convolution.
Here we only need one-dimensional boxed convolution. 
See \cite[Lecture 17, 18]{nica2006lectures} for the detail.
Let $n \in \N$ and $\pi \in \NC(n)$. 
Write  $[n] = \{ 1, 2, \dots, n \}$ and consider the discriminant union $[n] \coprod [n]$.  We write the elements from the second entry as $\Bar{k} \ (k \in [n])$, and write $[\Bar{n}]= \{ \Bar{1}, \Bar{2}, \dotsm \Bar{n}\}$.
We define an order as follows:
\begin{align}
1 \leq  \Bar{1} \leq 2 \leq \Bar{2} \dots \leq n \leq \Bar{n}.
\end{align}
Then the set $[n] \coprod [n]$ is a totally ordered set. 
Let $\pi \in \NC(n)$ and 
\begin{align}
J:= \{ \rho \in \NC([\Bar{n}]) \mid \pi \cup \rho \in \NC([n] \cup [n]) \}.
\end{align}
Then $J$ has the biggest element with respect to the following partially order of $\NC(n)$: for $\rho$ and $\pi \in \NC(n)$,  $\rho \leq \pi $ if $\forall V_1, V_2 \in \rho, \exists W \in \pi$ such that $V_1 \cup V_2 \subset W$.
The \emph{Kreweras complement} of $\pi$, denoted by $K(\pi)$ is defined as
\begin{align}
K(\pi)  := \max J.\label{align:Kreweras}
\end{align}
%
%boxed convolution
For $n \in \N$ and $\NC(n)$,  we denote
\begin{align}
\coef{f}{n}{;\pi}:= \prod_{V \in \pi}\coef{f}{\abs{V}}{},
\end{align}
where $\abs{V}$ is the number of elements in $V$.
For $f, g \in \fpower$, the \emph{one dimensional boxed convolution} (boxed convolution, for short), denoted by $f \boxconv g $ is defined as
\begin{align}
(f \boxconv g ) (z) := \sum_{m=1}^\infty \sum_{\pi \in \NC(m)} \coef{f}{n}{;\pi}\coef{g}{n}{;K(\pi)}  z^m,
\end{align}
where $K(\pi)$ is the Kreweras complement \eqref{align:Kreweras}.
One has the operation $\boxconv$ is associative and commutative \cite[Proposition~17.5, Corollary~17.10]{nica2006lectures}.
In addition, let us denote by $\Delta$ the series in $\fpower$ defined as
\begin{align}
\Delta(z) = z.
\end{align}
Then $\Delta$ is the unit of $(\fpower, \boxconv)$ \cite[Proposition~17.5]{nica2006lectures}.
We denote by $\fpower^\times$ the set of invertible elements in $\fpower$ with respect to $\boxconv$.
For $f \in \fpower$, we denote by $f^{-1}$ its  inverse with respect to $\boxconv$.
Then by \cite[Proposition~17.7]{nica2006lectures},
\begin{align}
\fpower^\times = \{ f \in \fpower  \mid \coef{f}{1}{} \neq 0 \}.
\end{align}

Third, we define the \emph{Zeta function} as 
\begin{align}
\fzeta(z) := \sum_{n=1}^{\infty} z^n.
\end{align}
Clearly $\fzeta \in \fpower^{\times}$.
Then we define the R-transform of formal power series.
\begin{defn}(R-transform)
	Let $f \in \fpower$. Let us define the $\emph{R-transform}$ of $f$ as
	\begin{align}
	R_f :=   f  \boxconv \fzeta^{-1}.
	\end{align}
\end{defn}

For any probability measure $\mu$ on $\R$ with all moments finite, we denote by $M_\mu$  its moment formal power series:
\begin{align}
M[\mu](z) = \sum_{n=1}^\infty m_n(z) z^n.
\end{align}
Let $(\mf{A}, \phi)$ be a C$^*$-probability
space, and let  $a$ be an element of $\mf{A}$.
The \emph{moment  power series} of $a$, denote by $M_a$, is a formal power series  defined as 
\begin{align}
M[a](z) = \sum_{n=1}^\infty \phi(a^n)z^n.
\end{align}    
We simply write 
\begin{align}
R[\mu] &= R_{M[\mu]},\\
R[a] &= R_{M[a]}. \label{align:r-compati}
\end{align}

Usually R-transform of  $a \in \mf{A}$ is defined as formal power series whose  coefficients are free cumulants (see \cite{nica2006lectures}). The compatibility   of our definition \eqref{align:r-compati} and usual definition is proven in \cite[Proposition~17.4]{nica2006lectures}.
In addition, the following holds.
\begin{lem}\label{lem:R[ab]}
	Let $(\mf{A}, \phi)$ be a C$^*$-probability space and  $a, b \in \mf{A}$.
	Assume that $(a, b)$ is free.
	Then 
	\begin{align}
	R[ab] = R[a] \boxconv R[b].
	\end{align}
\end{lem}
\begin{proof}
	This is a direct consequence of \cite[Proposition~17.2]{nica2006lectures}.
\end{proof}
Lastly, note that  it holds that  for $f \in \fpower$,
\begin{align}
f \in \fpower^{\times} \text{   if and only if  } R_f \in \fpower^{\times},
\end{align}
since $\coef{R_f}{1}{} = \coef{f}{1}{}$.
Now we have prepared to define the free multiplicative deconvolution.

\begin{defn}(free multiplicative deconvolution)
	For $f \in \fpower$ and $g \in \fpower^\times$,  the \emph{free multiplicative deconvolution} of $f$ with $g$ is defined as  
	\begin{align}
	f \boxdiag g :=  (R_f \boxconv R_g^{-1} ) \boxconv \fzeta.
	\end{align}
	Equivalently, $f \boxdiag g$ is the unique formal power series in $\fpower$
	determined by 
	\begin{align}
	R_f =  R_g \boxconv R_{(f\boxdiag g)}.
	\end{align}
\end{defn}

\begin{example}
	Let $ \beta \in \R$ and 
	$\delta_{\beta}$  be the delta measure on $\R$ whose support is $\{ \beta\} \subset \R$.
	Then  
	\begin{align}
	M[\delta_\beta](z) &=   \sum_{n=1}^\infty \beta^n z^n = [\fzeta \boxconv (\beta \Delta) ] (z),
	\end{align}
	since
	\begin{align}
	\coef{\beta \Delta}{n}{;K(\pi)} =  \begin{cases}
	\beta^n &; \pi = \{ \{1,2, \dots, n \} \},\\
	0&;  \mathrm{ otherwise}.
	\end{cases}
	\end{align}
	Note that $K( \{ \{1,2, \dots, n \} \} ) = \{\{1\}, \{2\}, \dots, \{n\}  \}$.
	Hence 
	\begin{align}
	R[\delta_\beta] = \beta \Delta.
	\end{align}
	Then for any $f \in \Xi$, we have
	\begin{align}
	R_{f(\beta \ \cdot\ )}[z] =  \sum_{n=1}^{\infty}  \coef{R_f}{n}{} \beta^n z^n =  R_f \boxconv R[\delta_\beta] (z).
	\end{align}
	In particular, if  $f \in \fpower^\times$, it holds that
	\begin{align}\label{align:scalar-deconv}
	 f ( \beta \ \cdot  \ ) \boxdiag f =  M[\delta_\beta].
	\end{align}
    In the case $f = M[a]$ with $a \in \mf{A}$, it is easy to show that 
	\begin{align}
     M[\beta a] \boxdiag M[a] =  M[\beta] = M[\delta_\beta],
    \end{align}
    since each scalar is free from any element of $\mf{A}$.
\end{example}

\begin{defn}
	Let $f, g \in \fpower$. 
	Then their \emph{free additive convolution}, denoted by $f \boxplus g \in \fpower$,  is defined as 
	\begin{align}
	f \boxplus g := ( R_f + R_g ) \boxconv \fzeta.
	\end{align}
	Equivalently, $f \boxplus g$ is the unique formal power series in $\fpower$ determined by 
	\begin{align}
	R_{f \boxplus g } = R_f + R_g.
	\end{align}
\end{defn}

\begin{note}
	Let $(\mf{A}, \phi)$ be a C$^*$-probability space. Let $q \in \mf{A}$ be a non-zero projection, that is, $q=q^*=q^2$.
	Then 
	\begin{align}
	(q\mf{A}q, \frac{1}{\phi(q)}\phi)
	\end{align}
	becomes a C$^*$-probability space.
    For $a \in \mf{A}$,
	we denote by $ M^{q\mf{A}q}[qaq]$ the moment power series of $qaq \ (a \in \mf{A})$ in $(q\mf{A}q,  \phi(q)^{-1}\phi)$:
	\begin{align}
	M^{q\mf{A}q}[qaq]  = \sum_{n=1}^\infty \frac{1}{\phi(q)} \phi[ (q a q)^n ] z^n.
	\end{align}
\end{note}

\begin{prop}\label{prop:ryan}
	Let $(\mf{A}, \phi)$ be a C$^*$-probability space.
	Assume that $a,c,p \in \mf{A}$ satisfies the following conditions:
	\begin{enumerate}
		\item $a^*=a$,
		\item $c$ is a circular element, that is, 
		\begin{align}
		c = \sigma \frac{s_1 + i s_2}{ \sqrt{2}},
		\end{align}
		where $(s_1, s_2)$ is a pair of free standard semicircular elements
		in $(\mf{A}, \phi)$ and $\sigma \in \R$,
		\item $q$ is a projection, and,
		\item $( \{c, c^*\}, \{a, q\} )$ is a pair of free families.
	\end{enumerate}
	Set $\lambda := \phi(q)$ and 
	\begin{align}
	f_\lambda(z) := \sum_{n=1}^\infty \lambda^{n-1}z^n.
	\end{align}
	Then  we have 
	\begin{align} \label{align:ryan-moment}
	M^{q\mf{A}q}[q(a+c)^*(a+c)q ]\boxdiag f_\lambda  
	=  (M^{q\mf{A}q}[qa^*aq] \boxdiag f_\lambda) \boxplus  (  M^{q\mf{A}q}[qc^*cq] \boxdiag f_\lambda ).
	\end{align}
\end{prop}

\begin{proof}
	This is a direct consequence of \cite[Theorem~3.4]{ryan2multiplicative}.
\end{proof}

\subsubsection{Free Poisson Distribution}\hfill

The formal power series $f_\lambda$ in Proposition~\ref{prop:ryan} is 
R-transform of a free Poisson distribution. We review on the free Poisson distribution.

\begin{defn}(Free Poisson Distribution)
	Let $\lambda >0$, $\alpha \in \R$.
	Then the \emph{free Poisson distribution} with rate $\lambda$ and jump size $\alpha$ is defined as the probability measure on $\R$ determined by 
	\begin{align}
	R[\nu_{\lambda, \alpha}]  = \lambda\sum_{n=1}^\infty  \alpha^n z^n.
	\end{align}    
\end{defn}

%See \cite[Proposition~12.11, Definition~12.12]{nica2006lectures} for the detail.
Usually free Poisson law  is defined as the limit law of free version of law of small numbers \cite[Definition~12.12]{nica2006lectures}.
The compatibility between our definition and usual definition is given by 
\cite[Proposition~12.11]{nica2006lectures}.
Note that $\nu_{\lambda, \alpha}$ is, in fact, a compactly supported probability measure. 
Note that
\begin{align}
f_\lambda = R[\nu_{ \lambda^{-1},  \lambda}].
\end{align}

\begin{lem}\label{lem:R[qaq]}
	Let $(\mf{A}, \phi)$ be a C$^*$-probability space,  $a \in \mf{A}$, and  $q \in \mf{A}$ be a non-zero projection free from $a$. Then it holds that
	\begin{align}
	R^{q\mf{A}q}[qaq](z) =  \lambda^{-1} R[a](\lambda z ),
	\end{align} 
	where $\lambda:= \phi(q)$.\end{lem}
This is well-known, but for the reader's convenience, we sketch the proof. 
\begin{proof}
	Note that $M^{q\mf{A}q}[qaq] = \lambda^{-1}M[qaq].$
	By the tracial condition and Lemma~\ref{lem:R[ab]}, 
	\begin{align}
	M[qaq]  = M[aq] =  R[aq] \boxconv \fzeta = R[a] \boxconv R[q] \boxconv \fzeta%
	= R[a] \boxconv M[q] =  R[a] \boxconv (\lambda \fzeta).
	\end{align}
	By definition of the boxed convolution, we have
	\begin{align}
	R[a] \boxconv (\lambda \fzeta) (z)  =  \sum_{n=1}^\infty \sum_{\pi \in \NC(n)} \coef{R[a]}{n}{;\pi}  \lambda^{\#K(\pi)}  z^n.
	\end{align}
	Since $\#\pi + \#K(\pi) = n+1$, this is equal to
	\begin{align}
	\lambda \sum_{n=1}^\infty \sum_{\pi \in \NC(n)} \coef{\frac{1}{\lambda}R[a]}{n}{;\pi}  \lambda  z^n = \lambda[  (\lambda^{-1}R[a](\lambda \ \cdot \ ) ) \boxconv  \fzeta] (z).
	\end{align}
	Thus
	\begin{align}
	M^{q\mf{A}q}[qaq] &= (\lambda^{-1}R[a]( \lambda \ \cdot \ ) ) \boxconv  \fzeta, \\
	R^{q\mf{A}q}[qaq] &= \lambda^{-1}R[a]( \lambda \ \cdot \ ).
	\end{align}
	
\end{proof}

\begin{example}
    Let $q,c \in \mf{A}$ and  $q$ be a nonzero-projection.
    Assume that $(\{q\}, \{c,c^*\} )$  is free pair in $(\mf{A}, \tau)$ and $c$ is a standard circular element.
	Then by Lemma~\ref{lem:R[qaq]}, 
	\begin{align}
	R^{q\mf{A}q}[qc^*cq](z) = \lambda^{-1}R[c^*c]( \lambda z) = 
	\lambda^{-1}\sum_{n=1}^{\infty}\lambda^n z^n = 
	R [\nu_{ \lambda^{-1},  \lambda} ](z) = f_\lambda(z).
	\end{align}    
	
\end{example}

\subsubsection{Second Lemma}\hfill

In this section, we convert the model to an operator of the form $qaq$ where $q$ is a projection.
Let $(\mf{A}, \phi)$ be a C$^*$-probability space.
Let $p, d \in \N$ with $p \geq d$ and write $\lambda = d/p$.
In this section and in next one, we denote by $C^{p,d}$ be a $p \times d$ matrix of $*$-free circular elements with 
\begin{align}
\phi[ (C^{p,d}_{ij} )^* C^{p,d}_{ij} ]= \frac{1}{d}.
\end{align}
Recall that
\begin{align}
W^\Box_\SPN(A,\sigma ) &= \left( A + \sigma C^{p,d} \right)^* \left( A + \sigma C^{p,d} \right).
\end{align}
Now we identify $C^{p,d}$ with $d \times d$ upper-left corner of $C^{p,p}$
with a normalization as the following:
\begin{align}
C^{p,p}_{ij} =  \sqrt{\lambda} C^{p,d}_{ij}, \  \forall i \in \{1, 2, \dots,  p\}, \ \forall j \in \{1,2,\dots, d \}.
\end{align}
Recall that a family $\{ C^{p,p}_{ij} \mid \text{  $1 \leq i,j \leq p$ } \}$  is a  $*$-free family of circular elements such as 
\begin{align}
\phi[  (C^{p,p}_{ij})^*C^{p,p}_{ij})]  =  \frac{1}{p}.
\end{align}
We write
\begin{align}
\mf{C}:= M_p(\mf{A}), \ \tau := \tr_p \otimes \phi.
\end{align}
Then $C^{p,p}$ is a circular element in $(\mf{C}, \tau)$, and it is standard, that is, 
\begin{align}
\coef{ R_{ (C^{p,p})^*C^{p,p} } }{n}{} =  1.
\end{align}
We define a projection  $\Pi \in M_p(\C) \subset \mf{C}$ as 
\begin{align}
\Pi_{ij} =  \begin{cases}
1, & \mathrm{ if \ } i=j\leq d,\\
0, & \mathrm{ otherwise.}
\end{cases}
\end{align}
One has $\tau(\Pi) = \lambda$.
For a $p \times d$-matrix $A$, let us denote by $\tilde{A}$ be the $p \times p$-square matrix obtained by adding zeros to $A$;
\begin{align}
\tilde{A} := \begin{bmatrix}
A &  O_{p,d}
\end{bmatrix}.
\end{align}
Now by definition, we have
\begin{align}
\Pi \left(  \tilde{A} +  \frac{\sigma}{ \sqrt{\lambda} }C^{p,p}  \right)^*   \left(  \tilde{A} +   \frac{ \sigma }{\sqrt{ \lambda} } C^{p,p} \right) \Pi %
&=  \begin{bmatrix}
\left( A + \sigma C^{p,d} \right)^* \left( A + \sigma C^{p,d} \right) & O_{d, p-d } \\
O_{p-d,d} & O_{p-d, p-d} \\
\end{bmatrix}.
\end{align}
Therefore,  for any $m \in \N$, 
\begin{align}
\frac{1}{d}\Tr_d \otimes \phi \left[  W^\Box_\SPN(A, \sigma)^m  \right] %
&=  \frac{1}{\lambda} \frac{1}{p} \Tr_p \otimes \phi \left\{ \left[ \Pi \left(  \tilde{A} + \frac{\sigma}{ \sqrt{\lambda} } C^{p,p}  \right)^*   \left(  \tilde{A} +   \frac{\sigma}{ \sqrt{\lambda} } C^{p,p} \right) \Pi \right]^m \right\},\\
\tr_d \otimes \phi \left[  W^\Box_\SPN(A, \sigma)^m  \right] %
&=  \frac{1}{\tr_p \otimes \phi(\Pi) }  \tr_p \otimes \phi \left\{ \left[ \Pi \left(  \tilde{A} + \frac{\sigma}{ \sqrt{\lambda} } C^{p,p}  \right)^*   \left(  \tilde{A} +   \frac{\sigma}{ \sqrt{\lambda} } C^{p,p} \right) \Pi \right]^m \right\}.
\end{align}
Equivalently, we have
\begin{align}\label{align:reduced-moments}
M[ W^\Box_\SPN (A,\sigma )]  = M^{\Pi \mf{C}\Pi} \left[ \Pi \left(  \tilde{A} + \frac{\sigma}{ \sqrt{\lambda} } C^{p,p}  \right)^*   \left(  \tilde{A} + \frac{\sigma}{ \sqrt{\lambda} } C^{p,p} \right) \Pi \right].
\end{align}
Recall that 
\begin{align}
M^{\Pi \mf{C}\Pi}[ \Pi X \Pi ](z)  =  \sum_{n=1}^\infty \frac{1}{\tau(\Pi)} \tau[ (\Pi X \Pi)^n ] z^n.
\end{align}

\begin{lem}\label{lem:scalar-deconv}
	Let $\alpha \in \R$. Then 
	\begin{align}
	M^{\Pi \mf{C}\Pi}[ \alpha \Pi   (C^{p,p} )^*  C^{p,p}   \Pi] ( z) \boxdiag M[\nu_{ \lambda^{-1}, \lambda}] = M[\delta_{\alpha}],
	\end{align}
	where $\delta_{\alpha}$ is the delta measure on $\R$ whose support is $\{\alpha\}$.
\end{lem}

\begin{proof}
	Now  $\{ C^{p,p} \}$ and $\{ \tilde{A}, \Pi \}$ is $*$-free in $(\mf{C}, \tau)$, since the entries of $A$  and $\Pi$ are scalar.
	By Lemma~\ref{lem:R[qaq]},
	\begin{align}
	R^{\Pi \mf{C}\Pi}[ \Pi   (C^{p,p} )^*  C^{p,p}   \Pi ] ( z)&=%
	\lambda^{-1}R[   (C^{p,p} )^*  C^{p,p} ] (\lambda z)\\
	&=%
	\lambda^{-1} \sum_{n=1}^\infty  (\lambda z)^{n}\\
	&=
	R[\nu_{ \lambda^{-1}, \lambda}].
	\end{align}
	Hence by \eqref{align:scalar-deconv}, 
	\begin{align}
    M^{\Pi \mf{C}\Pi}[ \alpha \Pi   (C^{p,p} )^*  C^{p,p}   \Pi] ( z) \boxdiag M[\nu_{ \lambda^{-1}, \lambda}] = 	M[\nu_{ \lambda^{-1},  \lambda}]( \alpha \ \cdot \ ) \boxdiag 	M[\nu_{ \lambda^{-1} \lambda}] = M[\delta_\alpha].
    \end{align}

\end{proof}

\begin{comment}

\begin{lem} (necesaly?)
\begin{align}
\tau_p \left[  (\frac{1}{d} YY^* )^m  \right]  = \frac{d}{p}\tau_d \left[  (\frac{1}{d} Y^*Y )^m  \right] = \lambda\tau_d \left[  (\frac{1}{d} Y^*Y )^m  \right] .
\end{align}

\end{lem}
content...
\end{comment}

\begin{cor}\label{cor:deconv-spn}
	Let $p,d \in\N$, $A \in M_{p,d}(\C)$, $\sigma \in \R$.
	Assume that $p \geq d$ and set $\lambda := d/p$.
	Then
	\begin{align}
	M[ W^\Box_\SPN( A, \sigma ) ] \boxdiag f_\lambda   =  \left(   M[ A^*A]  \boxdiag f_\lambda \right) \boxplus M[\delta_{\sigma^2/\lambda}].
	\end{align}
	
\end{cor}

\begin{proof}
	By  \eqref{align:reduced-moments} and  Proposition~\ref{align:ryan-moment}, the left-hand side is equal to 
	\begin{align}
	\left( M^{\Pi \mf{C}\Pi} [ \Pi \tilde{A}^* \tilde{A} \Pi ]  \boxdiag f_\lambda \right) \boxplus  \left( M^{\Pi \mf{C}\Pi}[ \frac{\sigma^2}{\lambda} \Pi (C^{p,p})^* C^{p,p} \Pi ] \boxdiag f_\lambda \right).
	\end{align}
	Now 
	\begin{align}
	M^{\Pi \mf{C}\Pi} [ \Pi \tilde{A}^* \tilde{A} \Pi ] = \frac{1}{\tau(\Pi)} M[ \tilde{A}^* \tilde{A} ] = \frac{1}{\tau(\Pi)}   \frac{d}{p}M[A^*A] = M[A^*A]. 
	\end{align}
	By Lemma~\ref{lem:scalar-deconv}, it holds that 
	\begin{align}
	M^{\Pi \mf{C}\Pi}[ \frac{\sigma^2}{\lambda} \Pi (C^{p,p})^* C^{p,p} \Pi ] \boxdiag f_\lambda =  M[\delta_{ \sigma^2/\lambda } ].
	\end{align}
	Hence the assertion holds.
\end{proof}

\begin{lem}\label{lem:moment_delta}
	Assume that $\alpha,\beta \in \R$, and $f, g \in \fpower$ satisfy 
	\begin{align}\label{align:assume}
	f \boxplus M[\delta_{\alpha}] = g \boxplus M[\delta_{\beta}].
	\end{align}
	Then 
	\begin{align}\label{align:subtract-scalar}
	f \boxplus  M[\delta_{\alpha - \beta } ] = g.
	\end{align}
\end{lem}

\begin{proof}
	Apply $\boxdiag \fzeta$ to both hand side of \eqref{align:assume}, then
	\begin{align}
	R_f(z) + \alpha z  &=  R_g(z) +  \beta z,\\
	R_f(z) + (\alpha - \beta)z  &=  R_g(z).
	\end{align}
	Applying $\boxconv \fzeta$ to both hand side, we have \eqref{align:subtract-scalar}.
\end{proof}

Now we prove the second key lemma.
\begin{lem}\label{lem:key-comb}
	Let $p,d \in\N$, $\sigma, \rho \in \R$, and,  $A$ and $B\in M_{p,d}(\C)$.
	Assume that $\sigma^2\geq \rho^2$  and  
	\begin{align}
	\muSPNB{A}{\sigma} = \muSPNB{B}{\rho}.
	\end{align}
	Then 
	\begin{align}\label{align:reduced}
	\muSPNB{A}{\sqrt{\sigma^2 -\rho^2}} = \muSPNB{B}{0}.
	\end{align}
	
\end{lem}
\begin{proof}
	By Corollary~\ref{cor:deconv-spn} and the assumption, we have
	\begin{align}
	( M_{A^*A} \boxdiag f_\lambda ) \boxplus M[\delta_{\sigma^2/\lambda}]  = %
	( M_{B^*B} \boxdiag f_\lambda ) \boxplus M[\delta_{\rho^2 /\lambda}].
	\end{align}
	Thus by Lemma~\ref{lem:moment_delta}, it holds that
	\begin{align}
	( M[A^*A] \boxdiag f_\lambda ) \boxplus M[\delta_{(\sigma^2-\rho^2)/\lambda}]  = %
	M[B^*B] \boxdiag f_\lambda. 
	\end{align}
	By using Corollary~\ref{cor:deconv-spn} again, we have 
	\begin{align}
	M[\muSPNB{A}{ \sqrt{\sigma^2 - \rho^2}} ] \boxdiag f_\lambda =    M[\muSPNB{B}{0} ] \boxdiag f_\lambda.
	\end{align}
    Equivalently,
    \begin{align}
    R[\muSPNB{A}{ \sqrt{\sigma^2 - \rho^2}} ]  \boxconv R[f_\lambda]^{-1} =    R[\muSPNB{B}{0} ] \boxconv R[f_\lambda]^{-1}.
    \end{align}
    Applying $\boxconv R[f_\lambda] \boxconv \fzeta$ to the both hand sides, we have
    \begin{align}
    M[\muSPNB{A}{ \sqrt{\sigma^2 - \rho^2}} ]  =    M[\muSPNB{B}{0} ].     
    \end{align}
    Since any compactly supported probability measure is determined by its moments, the assertion holds.
\end{proof}

\subsubsection{Proof of Identifiability}\label{sssec:proof-of-iden}

\begin{proof}[proof of Theorem~\ref{thm:identifiability-spn}]
	Without loss of generality, we may assume that $\sigma^2 \geq \rho^2$.
	Let $\muSPNB{A}{\sigma} = \muSPNB{B}{\rho}$.  
    First, by Lemma~\ref{lem:key-comb}, we have 
	\begin{align}
	\mu^\Box(A,\sqrt{\sigma^2 - \rho^2}) =\mu^\Box(B,0).
	\end{align}
	Second, Lemma~\ref{lem:key-analy} implies $\sqrt{\sigma^2 - \rho^2} = 0$. Then $\mu_{A^*A} = \mu_{B^*B}$, which completes the proof.
\end{proof}

\section{Acknowledgement}
We would like to thank Hiroaki Yoshida for discussions.
We appreciate Yuichi Ike's valuable comments on our manuscript.

%\bibliographystyle{abbrv}
%\bibliographystyle{plain}
%\bibliography{reference_iden}

\begin{thebibliography}{10}

\bibitem{allman2009identifiability}
E.~S. Allman, C.~Matias, J.~A. Rhodes, et~al.
\newblock Identifiability of parameters in latent structure models with many
  observed variables.
\newblock {\em Ann. Stat.}, 37(6A):3099--3132, 2009.

\bibitem{amari2016information}
S.~Amari.
\newblock {\em Information geometry and its applications}.
\newblock Springer, 2016.

\bibitem{belinschi2013analytic}
S.~T. Belinschi, T.~Mai, and R.~Speicher.
\newblock Analytic subordination theory of operator-valued free additive
  convolution and the solution of a general random matrix problem.
\newblock {\em J. Reine Angew. Math.}, 2013.

\bibitem{burda2011applying}
Z.~Burda, A.~Jarosz, M.~A. Nowak, J.~Jurkiewicz, G.~Papp, and I.~Zahed.
\newblock Applying free random variables to random matrix analysis of financial
  data. part {I}: The gaussian case.
\newblock {\em Quant. Fin.}, 11(7):1103--1124, 2011.

\bibitem{couillet2011deterministic}
R.~Couillet, M.~Debbah, and J.~W. Silverstein.
\newblock A deterministic equivalent for the analysis of correlated mimo
  multiple access channels.
\newblock {\em IEEE Trans. on Inform. Theory}, 57(6):3493--3514, 2011.

\bibitem{hachem2012large}
W.~Hachem, P.~Loubaton, X.~Mestre, J.~Najim, and P.~Vallet.
\newblock Large information plus noise random matrix models and consistent
  subspace estimation in large sensor networks.
\newblock {\em Random Matrices Theory Appl.}, 1(02):1150006, 2012.

\bibitem{hasegawa2013random}
A.~Hasegawa, N.~Sakuma, and H.~Yoshida.
\newblock Random matrices by {MA} models and compound free poisson laws.
\newblock {\em Probab. Math. Statist}, 33(2):243--254, 2013.

\bibitem{hayase2018cauchy}
T.~Hayase.
\newblock Cauchy noise loss for stochastic optimization of random matrix models
  via free deterministic equivalents.
\newblock \href{https://arxiv.org/abs/1804.03154}{arXiv:1804.03154 [stat.ML]},
  2018.

\bibitem{Mingo2017free}
J.~A. Mingo and R.~Speicher.
\newblock {\em Free probability and random matrices}, volume~35.
\newblock Springer, 2017.

\bibitem{nica2006lectures}
A.~Nica and R.~Speicher.
\newblock {\em Lectures on the combinatorics of free probability}, volume~13.
\newblock Cambridge Univ. Press, 2006.

\bibitem{ryan2007free}
{\O}.~Ryan and M.~Debbah.
\newblock Free deconvolution for signal processing applications.
\newblock In {\em IEEE Trans. Inform. Theory}, pages 1846--1850. IEEE, 2007.

\bibitem{ryan2multiplicative}
{\O}.~Ryan and M.~Debbah.
\newblock Multiplicative free convolution and information-plus-noise type
  matrices.
\newblock \href{https://arxiv.org/abs/math/0702342}{ arXiv:math/0702342
  [math.PR]}, 2007.

\bibitem{speicher1998combinatorial}
R.~Speicher.
\newblock {\em Combinatorial theory of the free product with amalgamation and
  operator-valued free probability theory}, volume 627.
\newblock American Math. Soc., 1998.

\bibitem{speicher2012free}
R.~Speicher and C.~Vargas.
\newblock Free deterministic equivalents, rectangular random matrix models, and
  operator-valued free probability theory.
\newblock {\em Random Matrices Theory Appl.}, 1(02):1150008, 2012.

\bibitem{vallet2012improved}
P.~Vallet, P.~Loubaton, and X.~Mestre.
\newblock Improved subspace estimation for multivariate observations of high
  dimension: the deterministic signals case.
\newblock {\em IEEE Tran. Inf. Theory}, 58(2):1043--1068, 2012.

\bibitem{voiculescu1985symmetries}
D.~V. Voiculescu.
\newblock Symmetries of some reduced free product {C}$^*$-algebras.
\newblock In {\em Operator algebras and their connections with topology and
  ergodic theory}, volume 1132, pages 556--588. Springer, Berlin, 1985.

\bibitem{voiculescu1991limit}
D.~V. Voiculescu.
\newblock Limit laws for random matrices and free products.
\newblock {\em Invent. Math.}, 104(1):201--220, 1991.

\end{thebibliography}

\end{document}